\documentclass{amsart}

\usepackage{graphicx}
\usepackage{tikz}
\usetikzlibrary{matrix,arrows,decorations.pathmorphing}
\usepackage{tikz-cd}
\usepackage{hyperref}
\usepackage{amsthm}
\usepackage[mathscr]{euscript}
\usepackage{amssymb}
\usepackage{mathrsfs}

\numberwithin{equation}{section}

\newtheorem{theorem}[equation]{Theorem}
\newtheorem{lemma}[equation]{Lemma}
\newtheorem*{conjecture*}{Conjecture}
\newtheorem*{convention*}{Convention}
\newtheorem*{notation*}{Notation}
\newtheorem*{question*}{Question}

\theoremstyle{definition}

\newtheorem{proposition}[equation]{Proposition}

\theoremstyle{remark}
\newtheorem{remark}[equation]{Remark}

\usepackage{color}

\usepackage[all,cmtip]{xy}

\newcommand{\F}{\mathbb{F}}
\newcommand{\Q}{\mathbb{Q}}

\newcommand{\Z}{\mathbb{Z}}
\newcommand{\wt}{\widetilde}
\newcommand{\wh}{\widehat}
\newcommand{\xra}{\xrightarrow}

\newcommand{\del}{\partial}
\newcommand{\opn}{\operatorname}
\newcommand{\ind}{\opn{ind}}
\newcommand{\ord}{\opn{ord}}
\newcommand{\per}{\opn{per}}
\newcommand{\perv}{\underline{{\opn{per}}}}

\newcommand{\an}[1]{\langle{#1}\rangle}

\begin{document}

\title{On The Topological Period-Index Problem over 8-manifolds}

\author{Diarmuid Crowley}
\address{School of Mathematics and Statistics, The University of Melbourne, Parkville VIC 3010, Australia}
\curraddr{}
\email{dcrowley@unimelb.edu.au}


\thanks{}

\author{Xing Gu}
\address{School of Mathematics and Statistics, The University of Melbourne, Parkville VIC 3010, Australia}
\email{xing.gu@unimelb.edu.au}
\thanks{}

\author{Christian Haesemeyer}
\address{School of Mathematics and Statistics, The University of Melbourne, Parkville VIC 3010, Australia}
\email{christian.haesemeyer@unimelb.edu.au}
\subjclass[2000]{}

\date{}

\dedicatory{}


\begin{abstract}
We establish upper bounds of the indices of topological Brauer classes over a closed orientable $8$-manifolds. In particular, we verify the Topological Period-Index Conjecture (TPIC) for topological Brauer classes over closed orientable $8$-manifolds of order not congruent to $2\pmod{4}$.
In addition, we provide a counter-example which shows that the 
TPIC fails in general for closed orientable 8-manifolds.
\end{abstract}

\maketitle
\section{Introduction}\label{sec:intro}
Let $X$ be a topological space. An Azumaya algebra $P$ over $X$ of degree $r$ is a principal $PU_r$-bundle over $X$ (or equivalently, a sheaf of algebras over the continuous functions on $X$ locally isomorphic to a matrix algebra). To such a principal bundle, one may associate a class $\alpha$ in the topological Brauer group $\opn{Br}_{top}(X):=TH^3(X)$ (where we write $TA$ for the torsion in an abelian group $A$ and throughout $H^*(X)$ denotes the integral cohomology groups of $X$). The order of $\alpha$ in the topological Brauer group is traditionally called the period, denoted $\opn{per}(\alpha)$, and it divides the degree of $P$. If $X$ is a finite $CW$-complex, every class $\alpha\in \opn{Br}_{top}(X)$ arises in this fashion, and it is a natural question what the possible degrees of Azumaya algebras $P$ with class $\alpha$ are. Following Antieau and Williams in \cite{antieau2014period}, we call the greatest common divisors of these degrees the index of $\alpha$, denoted $\opn{ind}(\alpha)$. As shown in {\it loc.\;cit.}, the period and index of $\alpha$ have the same prime divisors; it follows that the index divides some power of the period. This insight gives rise to the following problem:

\begin{question*}[Topological Period-Index Problem]
Let $X$ be a topological space, homotopy equivalent to a finite CW-complex. Is there a bound $e = e(X)$,  computable in terms of ``easy" topological invariants of $X$, such that for every class $\alpha\in \opn{Br}_{top}(X)$ we have  $\opn{ind}(\alpha)|\opn{per}(\alpha)^e$?
\end{question*}

\begin{remark}
It follows from our hypotheses that the Brauer group of $X$ is finite, and thus there is {\it some} bound $e$ depending on $X$. The problem is to describe it using other topological invariants of $X$.
\end{remark}

 As explained in \cite{antieau2013topological}, analogy with a major conjecture (known as the ``Period-Index Conjecture", see \cite{Co}) in the study of central simple algebras over function fields leads to a naive guess, namely that for $X$ of even dimension $2d$, one can take $e(X) = d-1$. Antieau and Williams give an algebraic-topological formula for the index of a Brauer class on CW complexes of dimension at most $6$, and use it to prove that the naive guess fails for $d\geq 3$. On the other hand, the first author and Grant show in \cite{crowley2018topological} that the naive guess is true for compact $\opn{spin}^c$ manifolds of dimension $6$. With this in mind, we state the following not as a conjecture to be proved, but rather as a hypothesis to be considered for certain classes of spaces such as closed orientable manifolds, $\opn{spin}^c$ manifolds, complex projective manifolds, etc.. We will refer to this hypothesis as the Topological Period-Index Conjecture (TPIC) for finite CW complexes, closed orientable manifolds, $\opn{spin}^c$ manifolds, complex projective manifolds, etc..





\begin{conjecture*}[The Topological Period-Index Conjecture]
Let $X$ be a topological space homotopy equivalent to a finite CW complex of dimension $2d$, and let $\alpha\in\opn{Br}_{top}(X)$, then
\[\opn{ind}(\alpha)|\opn{per}(\alpha)^{d-1}.\]
\end{conjecture*}




In this paper we establish upper bounds of the indices of topological Brauer classes over closed orientable $8$-manifolds, and in particular verify the TPIC for topological Brauer classes over closed orientable $8$-manifolds of order not congruent to $2\pmod{4}$.

\begin{theorem}\label{thm:main}
Let $M$ be a closed orientable manifold of dimension $8$ and $\alpha\in TH^3(M)$ a topological Brauer class. Then we have
\begin{equation*} 
\begin{cases}
\opn{ind}(\alpha)|2\opn{per}(\alpha)^3, \textrm{ if }\opn{per}(\alpha)\equiv 2 \!\!\! \pmod{4},\\
\opn{ind}(\alpha)|\opn{per}(\alpha)^3, \textrm{ otherwise}.
\end{cases}
\end{equation*}
In particular, the TPIC holds if $\opn{per}(\alpha)\not\equiv 2 \! \pmod{4}$.
\end{theorem}

Classes $\alpha$ in the Brauer group of a space $X$ give rise to a notion of $\alpha$-twisted $K$-theory on $X$, see \cite{PMIHES_1970__38__5_0} and \cite{atiyah2005twisted}, and this twisted $K$-theory is computed by a twisted Atiyah-Hirzebruch spectral sequence, see \cite{atiyah2005twisted}. In \cite{antieau2014period}, Antieau and Williams show that (for a connected compact space $X$) the index of the class $\alpha$ can be computed as the index of the subgroup of permanent cycles for the $\alpha$-twisted spectral sequence in $H^0(X) = \mathbb{Z}$.
It follows that we may decompose $\opn{ind}(\alpha)$ as the product of orders of $G_i(\alpha)$ where $G_i$'s are certain (higher) cohomology operations obtained from the differentials of the spectral sequence, with $G_0$ the identity. Thus we obtain the ``period vector'' of $\alpha$
\[\perv(\alpha):=\bigl( \opn{ord}(G_0(\alpha)),\ \opn{ord}(G_1(\alpha)),\ \cdots,\ \opn{ord}(G_{m-2}(\alpha)) \bigr)\]
where $2m$ is the dimension of the ambient CW complex $X$. In other words, we decompose $\opn{ind}(\alpha)$ as the product of all entries of $\perv(\alpha)$. A more detailed discussion on period vectors will be presented in Section \ref{sec:prelim}. The proof of Theorem \ref{thm:main} proceeds via a careful study of these higher cohomology operations. In the universal case (where the space $X$ is a (skeleton of) an appropriate Eilenberg-MacLane space) the operations $G_1$ and $G_2$ are described in Section \ref{sec:prelim}, recalling results from \cite{antieau2013topological} and \cite{gu2017topological}.

Further consideration leads to additional restrictions on the entries of the period vectors which are not visible from previous results such as Theorem \ref{thm:tpip 8-cplex}. For instance, we have the following

\begin{theorem}\label{thm:not(2,2,4)}
Let $X$ be a finite connected CW complex with $\alpha\in TH^3(X)$ of period $2$.
Furthermore, suppose $\opn{ord}(G_1(\alpha))=2$. Then we have $\opn{ord}(G_2(\alpha))|2$.
\end{theorem}

In particular, the theorem implies that for a finite connected $8$-complex $X$ and a Brauer class $\alpha\in TH^3(X)$ with $\per(\alpha)=2$, it is not possible to have $\perv(\alpha)=(2,2,4)$, although
such a period vector would not violate Theorem
\ref{thm:tpip 8-cplex}, the main result on the TPIC for $8$-complexes.

As a partial complement of Theorem \ref{thm:main}, we have the following consequence of Theorem \ref{thm:not(2,2,4)}:

\begin{theorem}\label{cor:not(2,2,4)}
Let $X$ be a finite connected CW-complex and let $\alpha\in TH^3(X)$ be of period $2\pmod{4}$, such that $4\nmid\opn{ord}(G_1(\alpha))$. Then we have
\begin{equation*}
\begin{cases}
\opn{ind}(\alpha)|3\opn{per}(\alpha)^3,\quad 3|\opn{per}(\alpha),\\
\opn{ind}(\alpha)|\opn{per}(\alpha)^3,\quad 3\nmid\opn{per}(\alpha).
\end{cases}
\end{equation*}
In particular, if in addition we have $3\nmid n$, then the TPIC holds for the pair $(X,\alpha)$.
\end{theorem}

Indeed, it follows from Theorem 1.3 of \cite{antieau2017prime} that it suffices to consider the 
cases $\per(\alpha)=2$ and that $\per(\alpha)$ is odd, for which we apply Theorem \ref{thm:not(2,2,4)} and Theorem \ref{thm:main}, respectively. Taken together, Theorems \ref{thm:main} and \ref{cor:not(2,2,4)} imply that the TPIC will hold for Brauer classes on closed orientable $8$-manifolds, except possibly for classes with 
$2$-primary period vector $\perv(\alpha)=(2,4,2)$ and here we state our final main result.

\begin{theorem} \label{thm:egs}
\hfill
\begin{enumerate}
\item
There is a closed orientable non-spin$^c$ $8$-manifold $M$ with $\alpha \in TH^3(M)$ such that 
$\perv(\alpha) = (2, 4, 2)$.
\item
There is a closed parallelisable (hence almost complex) $8$-manifold $M$ with
$\alpha \in TH^3(M)$ such that 
$\perv(\alpha) = (2, 4, 1)$.
\end{enumerate}
\end{theorem}
\noindent
Part (1) of Theorem \ref{thm:egs} shows that the TPIC fails in general for closed orientable $8$-manifolds.
To the best of our knowledge, the TPIC remains open for spin$^c$ $8$-manifolds
and in particular for the analytic spaces underlying smooth complex varieties
of (complex) dimension $4$.  Part (2) of Theorem \ref{thm:egs} does show that
spin$^c$ $8$-manifolds can have interesting period vectors and that if the TPIC
holds for such manifolds, it is sharp in general.



The rest of this paper is organised as follows:
In Section 2 we develop some necessary preliminaries.
We prove Theorem \ref{thm:main} and in Section 3 and
Theorem \ref{cor:not(2,2,4)} in Section 4.
In Section 5 prove the existence of the examples
in Theorem 5.

\vskip 0.25cm

\noindent
{\bf Acknowledgements:} We would like thank Mark Grant for numerous
helpful comments and conversations.

\section{Preliminaries}
\label{sec:prelim}

\subsection{The topological period-index problem over $8$-complexes} \label{ss:tpip 8-cplex}

In \cite{gu2017topological} and \cite{gu2018topological}, the third author considered the topological period-index problem for finite connected CW-complexes of dimension $8$. Before stating the main result, we introduce the notation, which we use throughout the paper.
\begin{notation*}
Throughout this paper $n$ is a positive integer, 
$K_n$ denotes the Eilenberg-MacLane space $K(\mathbb{Z}/n,2 )$,
$\zeta_n'$ 
the fundamental class of $H^2(K_n;\mathbb{Z}/n)$,  $\beta^{\mathbb{Z}/n}$ the Bockstein homomorphism
\[H^*(-;\mathbb{Z}/n)\rightarrow H^{*+1}(-;\mathbb{Z}),\]
and $\zeta_n :=\beta^{\mathbb{Z}/n}(\zeta'_n)$.
In addition $\epsilon_p(n)$ denotes the greatest common divisor of a prime number $p$ and $n$.
\end{notation*}


\begin{theorem}[Theorem 1.3, \cite{gu2018topological}]\label{thm:tpip 8-cplex}
Let $X$ be a topological space with the homotopy type of a finite connected $8$-dimensional CW-complex, and let $\alpha\in TH^{3}(X)$
be a topological Brauer class of period $n$. Then
\begin{equation}\label{eq:bound}
\begin{cases}
\opn{ind}(\alpha)|\epsilon_{3}(n)n^3,\quad\textrm{if }4|n,\\
\opn{ind}(\alpha)|\epsilon_{2}(n)\epsilon_{3}(n)n^3,\quad\textrm{otherwise.}
\end{cases}
\end{equation}
In addition, if $X$ is the $8$-skeleton of $K_n$ and $\alpha=\zeta_n$, then
\begin{equation*}
\begin{cases}
\opn{ind}(\zeta_n)=\epsilon_{3}(n)n^3,\quad\textrm{if }4|n,\\
\opn{ind}(\zeta_n)=\epsilon_{2}(n)\epsilon_{3}(n)n^3,\quad\textrm{otherwise.}
\end{cases}
\end{equation*}
In particular, the sharp lower bound of $e$ such that $\opn{ind}(\alpha)|n^{e}$ for all $X$ and $\alpha$ is $4$.
\end{theorem}

\subsection{The Atiyah-Hirzebruch Spectral Sequence for twisted $K$-theory} \label{ss:TAHSS}

For a connected topological space $X$ and a class $\alpha\in TH^{3}(X)$, Donovan and Karoubi (\cite{PMIHES_1970__38__5_0}) and also Atiyah and Segal (\cite{atiyah2004twisted}, \cite{atiyah2005twisted}) considered the twisted versions of complex $K$-theory of $X$ with respect to $\alpha$, which we denote by $KU(X)_{\alpha}$, following the convention in \cite{antieau2014period}. This generalizes the usual complex topological $K$-theory $KU(X)$ in the sense that when $\alpha=0$, we have $KU(X)_{\alpha}\cong KU(X)$.

Similarly to untwisted $K$-Theory, there is a twisted Atiyah-Hirzebruch spectral sequence for
twisted $K$-Theory, $\widetilde{E}_{*}^{*,*}$, 
which converges to $KU(X)_{\alpha}$ when $X$ is homotopy equivalent a finite CW complex
and whose $E_2$-terms,
\begin{equation*}
\widetilde{E}_{2}^{s,t}\cong
\begin{cases}
H^{s}(X),\quad\textrm{if $t$ is even,}\\
0, \quad\textrm{if $t$ is odd,}
\end{cases}
\end{equation*}
do not depend on the twist.
For more details, see \cite{antieau2013topological}, \cite{antieau2014period} and \cite{atiyah2005twisted}. The spectral sequence is closely related to the index of $\alpha$, as shown in the following
\begin{theorem}[Theorem 3.1, \cite{gu2017topological}]\label{thm:AH diff}
Let $X$ be a connected finite CW-complex and let $\alpha\in\operatorname{Br}_{top}(X)$. Consider $\widetilde{E}_{*}^{*,*}$, the twisted Atiyah-Hirzebruch spectral sequence (AHSS) with respect to $\alpha$ with differentials $\widetilde{d}_{r}^{s,t}$ with bi-degree $(r, -r+1)$. In particular, $\widetilde{E}_{2}^{0,0}\cong\mathbb{Z}$, and any $\widetilde{E}_{r}^{0,0}$ with $r>2$ is a subgroup of $\widetilde{E}_{2}^{0,0}$ and therefore generated by a positive integer. 
The subgroups $\widetilde{E}_{3}^{0,0}$ and $\widetilde{E}_{\infty}^{0,0}$ are generated respectively 
by $\opn{per}(\alpha)$ and $\opn{ind}(\alpha)$.
\end{theorem}
\noindent
We call this spectral sequence the twisted 
AHSS for the pair $(X,\alpha)$.

Theorem \ref{thm:AH diff} is the background result used in the proof of Theorem \ref{thm:tpip 8-cplex}
and shows why understanding the differentials in the twisted AHSS is the key to computing $\ind(\alpha)$.
To begin, we consider the case whe $X$ is the universal space $K_n$.
The cohomology of $K_n$ can be deduced from \cite{Ca}, the standard reference for (co)homology of Eilenberg-MacLane spaces. Alternatively, all material in this section can be found in \cite{antieau2013topological} and \cite{gu2017topological}.

Recall that for any integer $r>1$ we have
\[H^5(K_n)\cong\mathbb{Z}/\epsilon_2(n)n,\]
and
\[H^7(K_n)\cong\mathbb{Z}/\epsilon_3(n)n.\]
Recall also that the canonical generators $Q_n$ of $H^5(K_n)$ and $R_n$ of $H^7(K_n)$ satisfy
\begin{equation}\label{eq:Pontryagin powers}
\begin{split}
\epsilon_2(n)Q_n=\beta^{\mathbb{Z}/n}((\zeta'_n)^2),\\
\epsilon_3(n)R_n=\beta^{\mathbb{Z}/n}((\zeta'_n)^3).
\end{split}
\end{equation}

\begin{theorem}\label{thm:tpip 8-cplex'}
Let $\zeta_n\in TH^3(K_n)$ be the canonical generator of $H^3(K_n)$. Then the twisted Atiyah-Hirzebruch spectral sequence for $(K_n,\zeta_n)$ satisfies
\begin{equation}\label{eq:d5 for K_n}
\widetilde{d}_5^{0,0}(n)=\lambda_1Q_n
\end{equation}
and
\begin{equation}\label{eq:d7 for K_n}
\widetilde{d}_7^{0,0}(\epsilon_2(n)n^2)=
\begin{cases}
\frac{\lambda_2}{2}R_n & 4|n,\\
\lambda_2R_n & \textrm{ otherwise},
\end{cases}
\end{equation}
where $\lambda_1$ and $\lambda_2$ are integers co-prime to $n$.
\end{theorem}

\begin{proof}
Equation (\ref{eq:d7 for K_n}) is Theorem B of \cite{antieau2013topological}. One then readily deduces (\ref{eq:d7 for K_n}) from Theorem \ref{thm:tpip 8-cplex} and Theorem \ref{thm:AH diff}.
\end{proof}

\subsection{Period Vectors}
As mentioned in the introduction, the finiteness of $X$ indicates that we may obtain $\opn{ind}(\alpha)$ by computing the differentials $\widetilde{d}_{r}^{0,0}$ for successive $r$'s. By doing so we obtain a stable chain of subgroups of $\mathbb{Z}$
\begin{equation*}
\widetilde{E}_2^{0,0}=\Z\supset\opn{per}(\alpha)\Z=\opn{Ker}\widetilde{d}_{3}^{0,0}
\supseteq\opn{Ker}\widetilde{d}_{5}^{0,0}\supseteq\cdots\supseteq\opn{Ker}\widetilde{d}_{2r+1}^{0,0}\supseteq\cdots.
\end{equation*}
Notice that, by Bott periodicity, the differentials $\widetilde{d}_r^{0,0}$ for even $r$ are trivial. If we let $a_r'$ ($r\geq 0$) be the unique positive integer generating $\opn{Ker}\widetilde{d}_{2r+3}^{0,0}$, then, in particular, we have $a_0'=\opn{per}(\alpha)$. We also fix the notation $a'_{-1}=1$. Consequently, we have $a_r'|a_{r+1}'$. The element  $\widetilde{d}_{2r+3}^{0,0}(a_{r-1}')$ is then a torsion element of order $a_r:=a_r'/a_{r-1}'$. Therefore, if $X$ is of even dimension $2m$, then we have
\begin{equation}\label{eq:a_r}
\opn{ind}(\alpha)=a_0a_1\cdots a_{m-2}.
\end{equation}

It is both conceptually illuminating and practically helpful to consider the higher cohomology 
operations $G_r$ defined on $TH^3(X)$ by
\begin{equation}\label{eq:Gr definition}
G_r(\alpha)=\widetilde{d}_{0,0}^{2r+3}(\perv(a_{r}')).
\end{equation}
By definition, $G_0$ is the identity on $TH^3(X)$.
Antieau and Williams \cite{antieau2013topological} determined 
the secondary operation $G_1$ (which they denote by $G$).
Specifically, if $\per(\alpha) = n$
and $\alpha = \beta^{\Z/n}(\xi)$ for $\xi \in H^2(X; \Z/n)$, then 
\cite[Theorem 5.2]{antieau2013topological} states that
\begin{equation} \label{eq:G_1}
G_1(\alpha) = [\lambda \beta^{\Z/\epsilon_2(n)n}(P_2(\xi))] \in H^5(X)/\alpha H^2(X),
\end{equation}
where $P_2(\xi) = \xi^2$ if $n$ is odd, $P_2(\xi)$
is the Pontrjagin square of $\xi$ if $n$ is even and $\lambda$ is an integer prime to $n$.
In general the computation of $G_2(\alpha)$ is more difficult that
the computation of $G_1(\alpha)$.  This because 
the definition of $G_2(\alpha)$ depends on the value of $G_1(\alpha)$
and because $G_2(\alpha)$ lies in a more complicated quotient than $G_1(\alpha)$.
However, for $\alpha \in H^3(K_n)$ the canonical generator, 
Theorem \ref{thm:tpip 8-cplex'} computes $G_2(\alpha)$.

The element $G_r(\alpha)$ is defined so that equation (\ref{eq:a_r}) becomes
\begin{equation*}
\opn{ind}(\alpha)=
\opn{ord}(G_0(\alpha))\opn{ord}(G_1(\alpha))\cdots\opn{ord}(G_{m-2}(\alpha)),
\end{equation*}
where $\opn{ord}(G_r(\alpha))$ is the order of $G_r(\alpha)$. Furthermore, we define the \emph{period vector} of $\alpha$, denoted $\perv(\alpha)$, to be the ordered sequence of non-negative integers
\[\perv(\alpha):= 
\bigl( \opn{ord}(G_0(\alpha)),\ \opn{ord}(G_1(\alpha)),\ \cdots,\ \opn{ord}(G_{m-2}(\alpha)) \bigr)\]
and refer to $\opn{ord}(G_r(\alpha))$ \emph{the $r$th period} of $\alpha$.

In terms of the period vector, what Antieau and Williams show is essentially the following: Let $X$ be the $6$th skeleton of $K_n$ and $\alpha$ be the pull-back of $\zeta_n\in H^3(K_n)$, then we have
\begin{equation}\label{eq:perv of 6 skeleton}
\perv(\alpha)=(n,\epsilon_2(n)n).
\end{equation}
Similarly, what the third author shows in \cite{gu2017topological} and \cite{gu2018topological} is essentially the following: Let $X$ be the $8$th skeleton of $K_n$ and $\alpha$ be the pull-back of $\zeta_n\in H^3(K_n)$, then we have
\begin{equation}\label{eq:perv of 8 skeleton}
\perv(\alpha)=
\begin{cases}
(n,2n,\epsilon_3(n)n), \quad n\equiv 2\pmod{4},\\
(n,n,\epsilon_3(n)n), \textrm{ otherwise.}
\end{cases}
\end{equation}
It follows from Theorem \ref{thm:AH diff} and the functoriality
of the twisted $K$-theory AHSS, that for a connected finite CW complex $X$ of dimension $2d$, and a class $\alpha\in TH^3(X)$ with period vector
\[\perv(\alpha):= \bigl( n,\ \opn{ord}(G_1(\alpha)),\ \cdots,\ \opn{ord}(G_{d-2}(\alpha)) \bigr),\]
we have
\begin{equation}\label{eq:divisibility}
\opn{ind}(\alpha)=
n\opn{ord}(G_1(\alpha))\cdots\opn{ord}(G_{d-2}(\alpha))|
n \opn{ord}(G_1(\zeta_n))\cdots\opn{ord}(G_{d-2}(\zeta_n)),
\end{equation}
where $\zeta_n$ is the canonical generator of $H^3(K_n)$. However, in general, a period vector needs to satisfy  more than (\ref{eq:divisibility}), as illustrated in Theorem \ref{thm:not(2,2,4)}, which we prove in Section \ref{sec:main}.

\subsection{The linking pairing of a manifold}\label{ss:linking}

In this subsection we recall the linking pairing of a closed, oriented, connected $m$-manifold $M$.
For $m=8$ and $x \in H^2(M; \Z/3)$ we use the linking pairing of $M$ to define a symmetric
cubic form on $V : = TH^2(M; \Z) \otimes \Z/3$.
The material in this subsection is a cubic version of parts of \cite[\S 3]{crowley2018topological}.


Recall that for a finite abelian group $G$, its torsion dual is defined as
\[G^{\wedge}:=\opn{Hom}(G,\mathbb{Q}/\mathbb{Z}).\]
If $H$ is another finite abelian group, then then a bilinear pairing
\[\lambda:G\times H\rightarrow\mathbb{Q}/\mathbb{Z}\]
gives rise to the adjoint homomorphisms
\[\widehat{\lambda}_l:G\rightarrow H^{\wedge},\quad g\mapsto\lambda(g,-)\textrm{ and }\ \widehat{\lambda}_r:H\rightarrow G^{\wedge}, \quad h\mapsto\lambda(-,h).\]

Given a topological space $X$, let
\[\beta^{\mathbb{Q}/\mathbb{Z}}: H^*(X;\mathbb{Q}/\mathbb{Z})\rightarrow TH^{*+1}(X;\mathbb{Z})\]
be the Bockstein homomorphism.
%
For $M$ as above, the linking pairing of $M$ is the bilinear map
$$ b^k_M \colon TH^k(M) \times TH^{m-k+1}(M) \to \Q/\Z,
\quad (x, y) \mapsto \an{\wt x y, [M]},$$
where $\wt x \in H^{k-1}(M; \Q/\Z)$ is such that $\beta^{\Q/\Z}(\wt x) = x$.
It is well-known that $b_M$ is {\em perfect},
meaning that $(\wh b^k_M)_l$ and $(\wh b^k_M)_r$ are both
isomorphisms; see \cite[\S 3]{crowley2018topological}.
For any positive integer $l$, the commutative diagram of short exact coefficient sequences
\begin{equation*}
 \xymatrix{
\Z \ar[d]_{=} \ar[r]^{\times l} &
\Z \ar[d]^{\times\frac{1}{l}} \ar[r] &
\Z/l \ar[d]^{\iota_l} \\
\Z \ar[r] &
\mathbb{Q} \ar[r] &
\mathbb{Q}/\Z, }
\end{equation*}
where $\iota_l:\mathbb{Z}/l \hookrightarrow\mathbb{Q}/\mathbb{Z}, 1\mapsto [\frac{1}{l}]$,
gives that for all classes $x \in H^{k-1}(M; \Z/l)$ and $y \in TH^{m-k+1}(M)$ we have
\begin{equation} \label{eq:beta}
 b^k_M(\beta^{\Z/l}(x), y) = \iota_l(\an{xy, [M]}).
\end{equation}

Assume now that $M$ is an $8$-manifold.
For $x \in H^2(M; \Z/3)$ we define the following trilinear homomorphism
\[\lambda_x: TH^2(M)\times TH^2(M)\times TH^2(M)\rightarrow\mathbb{Q}/\mathbb{Z},
\quad (z_1,z_2,z_3)\mapsto \iota_3(\an{xz_1z_2z_3, [M]}).\]
By \eqref{eq:beta}
\[ \lambda_x(z_1, z_2, z_3) = b_M^3(\beta^{\Z/3}(x), z_1z_2z_3) .\]
For $V:=TH^2(M)\otimes\mathbb{Z}/3$ the symmetric trilinear homomorphism $\lambda_x$
induces the map
\[\lambda^V_x: V\times V\times V \rightarrow\mathbb{Z}/3,
\quad ([z_1], [z_2], [z_3]) \mapsto \iota_3^{-1}(\lambda_x(z_1, z_2, z_3),\]
which is a symmetric trilinear function of $(\mathbb{Z}/3)$-vector spaces.
In the next subsection, we will make some elementary remarks about symmetric
trilinear forms such as $\lambda_x^V$.
%
%

\subsection{Some elementary cubic algebra} \label{ss:cubic}
In this subsection we establish elementary algebraic facts about symmetric trilinear forms.
This generalises analogous material from \cite{crowley2018topological} on symmetric bilinear forms.

Let $\F_3 = \Z/3$ be the field with three elements.
A symmetric trilinear form on a finite dimensional $\F_3$ vector space $V$ is a trilinear map
$$ \lambda \colon V \times V \times V \to \F_3 $$
such that for any permutation $\sigma \in \Sigma_3$ we have
$$ \lambda(v_1, v_2, v_3) = \lambda(v_{\sigma(1)}, v_{\sigma(2)}, v_{\sigma(3)}).$$
Given trilinear forms $(V, \lambda_1)$ and $(W, \lambda_2)$, their orthogonal sum
is the trilinear form
$$ \lambda \colon (V \oplus W) \times (V  \oplus W) \times (V \oplus W) \to \F_3$$
define by
$$ (\lambda_1 \oplus \lambda_2) \bigl( (v_1, v_2, v_3), (w_1, w_2, w_3) \bigr) =
\lambda_1(v_1, v_2, v_3) + \lambda_2(w_1, w_2, w_3).$$

Let $V^* = \mathrm{Hom}(V, \F_3)$ denote the dual of $V$.
We consider two adjoint maps associated to a symmetric trilinear form $(V, \lambda)$.
The first of these is the homomorphism
$$ \wh \lambda^1 \colon V \to V^* \times V^*,
\quad u \mapsto ((v, w) \mapsto \lambda(u, v, w)).$$
We call the kernel of $\lambda_1$  the {\em radical} of $\lambda$ and denote it by $R$, since it
is the subspace of $V$ which combines trivially with every pair in $V \times V$.
Writing $V = R \oplus V/R$, we notice that
$\lambda$ induces a symmetric trilinear form $\bar \lambda$ on $V/R$ and that
$\lambda$ is isomorphic to the orthogonal sum of the zero form on $R$
and $\lambda_{R}$:
\begin{equation} \label{eq:trilinear_sum}
 \lambda \cong (R,0) \oplus (V/R, \bar \lambda)
\end{equation}
If $\lambda^1$ is injective, we call $\lambda$ {\em nondegenerate}.
In particular, $\bar \lambda$ is nondegenerate.

The second adjoint of $(V, \lambda)$ is the homomorphism
$$ \widehat{\lambda}^2\colon V \times V \to V^*,
\quad (u, v) \mapsto (w \mapsto \lambda(u, v, w)).$$
Notice that the dual of $\widehat{\lambda}^2$ is precisely $\widehat{\lambda}_1$:
$(\widehat{\lambda}^2)^* = \widehat{\lambda}^1$.
It follows that $\widehat{\lambda}^2$ is onto if and only if $\widehat{\lambda}^1$ is injective;
i.e.~$\lambda$ is nondegenerate.

We define the {\em characteristic element} of $(V, \lambda)$ to be the map
\[ \gamma(\lambda) \colon V \to \F_3,
\quad v \mapsto \lambda(v, v, v).\]
The symmetry of $\lambda$ ensures that $\gamma(\lambda)$ is linear;
i.e.\ we have the following equations: For all $v, w \in V$
\begin{eqnarray*}
\gamma(\lambda)(v+w) & = & \lambda(v+w, v+w, v+w) \\
& = & \lambda(v, v, v) + \lambda(v, v, w) + \lambda(v, w, v) + \lambda(w, v, v) + \\
& & \lambda(w, w, v) + \lambda(w, v, w) + \lambda(v, w, w) + \lambda(w, w, w)
\\
& = & \lambda(v, v, v) + 3\lambda(v, v, w) + 3\lambda(w, w, v) + \lambda(w, w, w) \\
& = & \lambda(v, v, v) + \lambda(w, w, w) \\
& = & \gamma(\lambda)(v)+\gamma(\lambda)(w),
\end{eqnarray*}
and for all $a \in \F_3 $
\begin{equation*}
\gamma(\lambda)(av)=\lambda(a v, a v, av) =
a^3 \lambda(v, v, v)=a\lambda(v, v, v)=a\gamma(\lambda)(v).
\end{equation*}
Hence $\gamma(\lambda) \in V^*$.
Moreover, it is clear that
the characteristic element of 
the orthogonal sum of trilinear forms, $(V_1, \lambda_1)$ and $(V_2, \lambda_2)$ is given by
\begin{equation} \label{eq:char_elt_sum}
\gamma(\lambda_1 \oplus \lambda_2) = (\gamma(\lambda_1), \gamma(\lambda_2)) \in V_1^* \oplus V_2^*
\end{equation}
The following lemma relates $\gamma(\lambda)$ to the second adjoint of $(V, \lambda)$.

\begin{lemma} \label{lem:diag3}
For all trilinear forms $(V, \lambda)$, we have $\gamma(\lambda) \in \opn{im}(\widehat{\lambda}^2)$.
\end{lemma}

\begin{proof}
By \eqref{eq:trilinear_sum}, we have that $(V, \lambda) = (R, 0) \oplus (V/R, \bar \lambda)$ is
the the sum of a zero form and the nondegenerate form $(V, \bar \lambda)$.
Hence $\gamma(\lambda) = (0, \gamma(\bar \lambda))$ by \eqref{eq:char_elt_sum}.
But nondegenerate trilinear forms are precisely those for which the second adjoint is onto.
Hence $\gamma(\bar \lambda) \in \opn{im}(\widehat{\bar \lambda}^2)$
and so $\gamma(\lambda) \in \opn{im}(\widehat{\lambda}^2)$.
\end{proof}


\subsection{The Pontrjagin cube}
The 
purpose of this section is to prove Lemma \ref{lem:R_3k}. Let $\alpha \in TH^3(X)$ with period $n$
and let

$$\beta^{\Z/n} \colon H^*(X; \Z/n) \to H^{*+1}(X) $$
be the mod~$n$ Bockstein, which lies is the exact sequence
$$ H^*(X; \Z/n) \xra{~\beta^{\Z/n}~}
H^{*+1}(X) \xra{~\times n~} H^{*+1}(X).
$$
As $\alpha$ has period $n$
we see that $\alpha = \beta^{\Z/n}(\xi)$ for some $\xi \in H^2(X; \Z/n)$.
In the case where $n = 3k$, where we have the Pontrjagin Cube
\[ P_3 \colon H^2(X; \Z/3k) \to H^6(X; \Z/9k), \]
which satisfies
\begin{equation}\label{eq:Pontryajin mod 3}
\rho_3(\xi^3)=\rho_3(P_3(\xi)),
\end{equation}
where $\rho_3$ denotes reduction modulo $3$. Recall that in Subsection \ref{ss:TAHSS} we defined classes $R_n\in H^7(K_n)$.
We use $R_n$ to denote the cohomology operation defined
by $R_n$ and this cohomology operation satisfies
\[ R_n(\xi) :=
\begin{cases}
\bigl[ \beta^{\Z/n}(\xi^3) \bigr]  & 3\nmid n, \\
\bigl[\beta^{\Z/3n}(P_3(\xi)) \bigr] & 3|n.
\end{cases}\]
Next we consider the following commutative diagram
\[ \xymatrix{
H^*(X; \Z/3k) \ar[d]^{\rho_3} \ar[r]^(0.6){\beta^{\Z/3k}} &
H^{*+1}(X) \ar[d]^{\times k} \\
H^*(X; \Z/3) \ar[r]^(0.575){\beta^{\Z/3}} &
H^{*+1}(X),
 } \]
%
whose commutativity
follows from the following commutative diagram of coefficient short exact sequences:
\begin{equation}\label{eq:Bockstein}
 \xymatrix{
\Z \ar[d]_{\times k} \ar[r]^{\times 3k} &
\Z \ar[d]^{=} \ar[r]^(0.425){\rho_{3k}} &
\Z/3k \ar[d]^{\rho_3} \\
\Z \ar[r]^{\times 3} &
\Z \ar[r]^(0.45){\rho_3} &
\Z/3 }
\end{equation}
Hence for all $x \in H^*(X; \Z/3k)$ we have the equation
\begin{equation}\label{eq:Bockstein multiple}
\beta^{\Z/3}(\rho_3(x)) = k\beta^{\Z/3k}(x),
\end{equation}
from which we obtain the following

\begin{lemma}\label{lem:R_3k}
For any space $X$ and $\xi\in H^2(X;\Z/3k)$, the cohomology operation $R_{3k}$ satisfies
\[3kR_{3k}(\xi)=\beta^{\Z/3}(\rho_3(\xi^3)).\]
\end{lemma}

\begin{proof}
This follows by direct computation, taking $k=3m$ in (\ref{eq:Bockstein}):
\begin{equation}\label{eq:R_3k}
\begin{split}
&3mR_{3m}(\xi)\\
=&3m\beta^{\Z/9k}(P_3(\xi)) \\
=&\beta^{\Z/3} \rho_3(P_3(\xi)) \\
=&\beta^{\Z/3}(\rho_3(\xi^3)). \\
\end{split}
\end{equation}
The last step follows from (\ref{eq:Pontryajin mod 3}).
\end{proof}

\subsection{Mod~3 Wu classes} \label{ss:Wu}
Let $p$ be a prime number and $M$ a closed (and oriented if $p>2$) manifold of dimension $n$. One readily verifies that the functions
\begin{equation*}
\begin{cases}
H^{n-r}(M;\mathbb{Z}/2)\rightarrow\mathbb{Z}/2,\quad u\mapsto
\opn{Sq}^r(u)\cap [M],\textrm{ if }p=2,\\
H^{n-2r(p-1)}(M;\mathbb{Z}/p)\rightarrow\mathbb{Z}/p,\quad u\mapsto
\mathscr{P}^{r}(u)\cap [M],\textrm{ if }p>2,
\end{cases}
\end{equation*}
are $\mathbb{Z}/p$-linear, where $\opn{Sq}^r$ and $\mathscr{P}^{r}$ are the Steenrod reduced power operations. By Poincar{\'e} duality, there are ``mod $p$ Wu classes'' $v_r\in H^r(M;\mathbb{Z}/2)$ and $v_r\in H^{2r(p-1)}(M;\mathbb{Z}/p)$ for $p>2$, such that
\begin{equation}\label{eq:Wu classes}
\begin{cases}
\opn{Sq}^r(u)=v_ru, \textrm{ if }p=2,\ u\in H^{n-r}(M;\Z/2),\\
\mathscr{P}^{r}(u)=v_ru, \textrm{ if }p>2,\ u\in H^{n-2r(p-1)}(M;\Z/p).
\end{cases}
\end{equation}
In \cite{Wu1955Pontrjagin}, Wu expressed the mod $2$ Wu classes as polynomials in Stiefel-Whitney classes of $M$. For $p>2$, Hirzebruch determined the mod $p$ Wu classes as polynomials in the Pontrjagin classes (Theorem 1, \cite{hirzebruch1953steenrod}). In particular, his work yields
\begin{proposition}\label{pro:mod 3 Wu classes}
When $p=3$, we have
\[v_1=\rho_3(p_1),\]
where $p_1$ is the first Pontrjagin class of $M$.
\end{proposition}

\section{Proof of Theorem \ref{thm:main}}
\label{sec:main}
In this section we prove Theorem \ref{thm:main}.
Recall that $M$ is a closed, connected, oriented manifold of dimension $8$
with first Pontrjagin class $p_1 \in H^4(M)$
and that $V$ is the $\Z/3$-vector space $V:=TH^2(M)/3 TH^2(M)$.
For any cohomology class $x\in H^2(M;\Z/3)$, we have the symmetric trilinear function
\[\lambda_x^V: V\times V\times V\rightarrow\Z/3,\quad ([z_1] , [z_2], [z_3])\mapsto \an{xz_1z_2z_3, [M]}
= b_M^3(\beta^{\Z/3}(x), z_1z_2z_3).
\]
For $_{3}TH^2(M)^\wedge \subseteq
TH^2(M)^\wedge$ the subgroup of $3$-torsion elements, we have the isomorphism
\[ \iota_3 \colon V^* \to  {_{3}TH^2(M)}^\wedge,
\quad f \mapsto (z \mapsto f(z)).\]
We also recall the following adjoint of the linking pairing of $M$,
which is the isomorphism
\[ (\wh b_M^7)_l \colon TH^7(M) \to TH^2(M)^\wedge,
\quad w \mapsto (z \mapsto b_M^7(w, z)).\]
\begin{lemma}\label{lem:cub to difference}
Let $M$ be a closed oriented manifold of dimension $8$ and $x\in H^2(M;\Z/3)$.
Then
\[ \iota_3(\gamma(\lambda^V_x))= (\wh b_M^7)_l(\beta^{\mathbb{Z}/3}(x)p_1-\beta^{\mathbb{Z}/3}(x^3)).\]
\end{lemma}
\begin{proof}
Fix the prime number $3$ and let $\mathscr{P}^r$ be the $r$th Steenrod reduced power operation.
For $z \in TH^2(M)$ the Cartan formula gives
\[\mathscr{P}^1(xz)= x^3z+xz^3.\]
On the other hand, by Proposition \ref{pro:mod 3 Wu classes} we have
\[\mathscr{P}^1(xz)= xp_1z.\]
Therefore, we have
\begin{equation}\label{eq:xz^3}
xz^3 = xp_1z-x^3z.
\end{equation}
Now we have
\begin{align*}
\iota_3(\gamma(\lambda_x^V))([z]) =& \iota_3(\langle xz^3, [M] \rangle) \\
=& \iota_3(\langle(xp_1z-x^3z, [M]\rangle) \quad \textrm{ (using (\ref{eq:xz^3}))}\\
=& b_M^7(\beta^{\Z/3}(x)p_1 - \beta^{\Z/3}(x^3), z) \\
=& (\wh b_M^7)_l(\beta^{\Z/3}(x)p_1 - \beta^{\Z/3}(x^3))(z).
\end{align*}
Since $[z] \in V$ is arbitrary, the required equation holds.
\end{proof}

\begin{lemma}\label{lem:Mod 3 relation}
Let $M$ be a closed oriented manifold of dimension $8$, and $x\in H^2(M;\Z/3)$. Then there are
$z_1, z_2\in H^2(M)$ such that
\[\beta^{\mathbb{Z}/3}(x^3)=\beta^{\mathbb{Z}/3}(x)(p_1-z_1z_2).\]
\end{lemma}

\begin{proof}
By Lemma \ref{lem:diag3} we have $z_1, z_2\in TH^2(M)$ satisfying
\[\gamma(\lambda^V_x) =  (\wh \lambda_x^V)_2([z_1], [z_2]) \]
and by definition, this means that
\[ \iota_3(\gamma(\lambda^V_x)) = (\wh b_M^7)_l(\beta^{\Z/3}(x)z_1z_2).\]
Now by Lemma \ref{lem:cub to difference} we have
\[ \iota_3(\gamma(\lambda^V_x)) = (\wh b_M^7)_l(\beta^{\mathbb{Z}/3}(x)p_1-\beta^{\mathbb{Z}/3}(x^3)). \]
Since $(\wh b_M^7)_l$ and $\iota_3 \colon V^* \to  {_{3}TH^2(M)}^\wedge$
are isomorphisms, we conclude that
\[ \beta^{\Z/3}(x)z_1z_2 = \beta^{\mathbb{Z}/3}(x)p_1-\beta^{\mathbb{Z}/3}(x^3) \]
and so
\[\beta^{\mathbb{Z}/3}(x^3)=\beta^{\mathbb{Z}/3}(x)(p_1-z_1z_2).\]
\end{proof}

Before stating Proposition \ref{pro:d3} we
recall the AHSS 
of the twisted $K$-theory of a pair $(X, \alpha)$, which we denote by 
$(\widetilde{E}_*^{*,*},\widetilde{d}_*^{*,*})$. We denote the untwisted one by $(E_*^{*,*}, d_*^{*,*})$. Then it follows from Bott periodicity theorem that we have
\begin{equation*}
E_2^{s,t}=\widetilde{E}_2^{s,t}\cong
\begin{cases}
H^s(X),\ \textrm{if $t$ is even},\\
0,\ \textrm{if $t$ is odd}.
\end{cases}
\end{equation*}
In view of the above identification we do not distinguish $H^s(X)$, $E_2^{s,t}$ and $\widetilde{E}_2^{s,t}$ for even $t$.  We also recall the cohomology operation $R_n\in H^7(K_n)$ of order $\epsilon_3(n)n$.

\begin{proposition}\label{pro:d3}
Let $M$ be a closed oriented $8$-manifold, 
$\alpha\in TH^3(M)$ with period $n$,
 $\xi\in H^2(M;\Z/n)$ such that $\beta^{\Z/n}(\xi)=\alpha$
 and $(\widetilde{E}_*^{*,*},\widetilde{d}_*^{*,*})$ be the 
AHSS for the twisted $K$-theory of the pair  $(M,\alpha)$. 
Identifying $H^7(M)$ with $\widetilde{E}_2^{7,-6}$, we have
\[nR_{n}(\xi)\in\opn{im}(\widetilde{d}_3^{4,-4}).\]
\end{proposition}

\begin{proof}
If $3\nmid n$, then we have $nR_n=0$ and there is nothing to show. Therefore we assume $n=3k$.

First we recall some general facts on the Atiyah-Hirzebruch spectral sequences of both the twisted and untwisted $K$-theories.  Then $(\widetilde{E}_*^{*,*},\widetilde{d}_*^{*,*})$ is a bi-graded module over $(E_*^{*,*}, d_*^{*,*})$. More precisely, suppose $u\in E_r^{s,t}$ and $v\in \widetilde{E}_r^{s',t'}$. Then the module structure yields  $uv\in\widetilde{E}_r^{s+s',t+t'}$ and furthermore,
\[\widetilde{d}_r(uv)=d_r(u)v+(-1)^{s+t}u\widetilde{d}_r(v).\]
It follows from Bott periodicity that $E_{2i}^{*,*}=E_{2i+1}^{*,*}$ for all $i>0$. For $u\in E_2^{*,*}=E_3^{*,*}$,
by the discussion in 2.3 of \cite{atiyah1961vector}, we have $d_3(u)=\opn{Sq}_{\Z}^3(u)$, where $\opn{Sq}^3_{\Z}$ be the $3$rd integral Steenrod square operation. On the other hand, by Proposition 2.4 of \cite{antieau2014period}, we have $\widetilde{d}_3(1)=\alpha$. Therefore, for $y\in\widetilde{E}_2^{s,t}$, we have
\begin{equation}\label{eq:d_3}
\begin{split}
&\widetilde{d}_3(y)=\widetilde{d}_3(y\cdot 1)\\
=&\opn{Sq}_{\Z}^3(y)+(-1)^{s+t}y\widetilde{d}_3(1)\\
=&\opn{Sq}_{\Z}^3(y)+(-1)^{s+t}y\alpha.
\end{split}
\end{equation}

It follows from Lemma \ref{lem:Mod 3 relation} that we may take $z_1, z_2\in H^2(M)$ such that
\[\beta^{\mathbb{Z}/3}(\rho_3(\xi^3))=(p_1-z_1z_2)\beta^{\mathbb{Z}/3}(\rho_3(\xi)).\]
By Lemma \ref{lem:R_3k} we have
\begin{equation}\label{eq:d_3'}
nR_n(\xi)=\beta^{\mathbb{Z}/3}(\rho_3(\xi^3))=(p_1-z_1z_2)\beta^{\mathbb{Z}/3}(\rho_3(\xi))=m(p_1-z_1z_2)\alpha,
\end{equation}
where the last equation follows from (\ref{eq:Bockstein multiple}).

On the other hand, we have  $\opn{Sq}^3_{\Z}(p_1)=0$ since by definition $p_1$ is the Chern class of some vector bundle over $M$, and $H^*(BU)$ concentrates in even degrees. Similarly we have  $\opn{Sq}^3_{\Z}(z_1z_2)=0$ since $H^*(K(\Z,2)^{\times2})$ concentrates in even degrees. Therefore we have
\begin{equation}\label{eq:d_3''}
\opn{Sq}^3_{\Z}(p_1-z_1z_2)=0.
\end{equation}

Finally, it follows from (\ref{eq:d_3}), (\ref{eq:d_3'}) and (\ref{eq:d_3''}) that we have
\[\widetilde{d}_3(m(p_1-z_1z_2))=nR_n(\xi),\]
and we conclude.
\end{proof}

\begin{proof}[Proof of Theorem \ref{thm:main}]
We can work in each connected component of $M$ individually, so we 
assume that $M$ is connected.
If $3\nmid n$, the theorem follows from Theorem \ref{thm:tpip 8-cplex'} formally. So we assume $3|n$.

Let $(\widetilde{E}_*^{*,*}, \widetilde{d}_*^{*,*})$ be the 
AHSS of the twisted $K$-theory of the pair $(X,\alpha)$, where $\opn{per}(\alpha)=n$. Let $\xi\in H^2(M;\Z/n)$ such that $\beta^{\Z/n}(\xi)=\alpha$.
By construction we have $\widetilde{E}_2^{0,0}\cong\Z$. For degree reasons there is no nontrivial differential into $\widetilde{E}_r^{0,0}$ for any $r\geq2$. We therefore identify $\widetilde{E}_r^{0,0}$ ($r\geq2$) as subgroups of $\widetilde{E}_2^{0,0}\cong\Z$, as in the setting of Theorem \ref{thm:AH diff},

It follows from (\ref{eq:d5 for K_n}) of Theorem \ref{thm:tpip 8-cplex'} that we have $\epsilon_2(n)n^2\in \widetilde{E}_7^{0,0}$. Furthermore, it follows from Proposition \ref{pro:d3} that the order of
$[R_n(\xi)]\in\widetilde{E}_7^{7,-6}$ divides $n$. Therefore, by (\ref{eq:d7 for K_n}) of Theorem \ref{thm:tpip 8-cplex'}, the order of $\widetilde{d}_7^{0,0}(\epsilon_2(n)n^2)$ divides $\frac{n}{2}$ if $4|n$ and divides $n$ otherwise. Hence we have
\begin{equation*}
\begin{cases}
\widetilde{E}_8^{0,0}\supseteq 2n^3\widetilde{E}_2^{0,0},\textrm{ if }n\equiv 2\pmod{4},\\
\widetilde{E}_8^{0,0}\supseteq n^3\widetilde{E}_2^{0,0},\textrm{ otherwise}.
\end{cases}
\end{equation*}
The theorem then follows from Theorem \ref{thm:AH diff}.
\end{proof}

\section{Proof of Theorem \ref{thm:not(2,2,4)}.} \label{sec:not(2,2,4)}

Recall that the cohomology group $H^5(K_2)$ is generated by an element $Q_2$ of order $4$. On the other hand, recall that the cohomology ring of $H^*(K(\Z,2))$ is isomorphic to the polynomial ring $\Z[\iota_2]$ with one generator $\iota_2$ of degree $2$.

Let ``$\times$'' denote the exterior cup product operation, then we have a class $2Q_2\times 1-\zeta_2\times\iota_2 \in H^5(K_2\times K(\mathbb{Z},2))$, which is represented by a homotopy class
\[2Q_2\times 1-\zeta_2\times\iota_2: K_2\times K(\Z,2)\rightarrow K(\Z,5).\]
Let $Y$ be the homotopy fiber of the above map and we have the following fiber sequence
\begin{equation}\label{eq:fiber seq}
Y\xrightarrow{h}K_2\times K(\Z,2)\xrightarrow{2Q_2\times 1-\zeta_2\times\iota_2} K(\Z,5).
\end{equation}
It then follows that
\[h^*: H^3(K_2\times K(\Z,2))\rightarrow Y\]
is an isomorphism of abelian groups. We take $\alpha_Y:=h^*(\xi_2\times 1)$ a generator of $H^3(Y)$. Then we have the following
\begin{lemma}\label{lem:universal Y}
Suppose $X$ is a connected finite CW complex and $\alpha\in H^3(X)$ satisfying $\per(\alpha)=2$ and $\opn{ord}(G_1(\alpha))=2$. Then there is a map $f: X\rightarrow Y$ satisfying $f^*(G_i(\alpha_Y))=G_i(\alpha)$, for $i=0,1,2$.
\end{lemma}
\begin{proof}
We take $f$ to be a map representing a cohomology class $\alpha'\in H^2(K_2,\Z/2)$ such that $\beta^{\Z/2}(\alpha')=\alpha$. The lemma then follows from (\ref{eq:Gr definition}).
\end{proof}

We verify some useful properties of $Y$ in the following
\begin{lemma}\label{lem:cohomology of Y}
\begin{enumerate}
\item
The map $h^*: H^i(K_2\times K(\Z,2))\rightarrow H^i(Y)$ is an isomorphism for $0\leq i\leq 3$. For $i=4$, we have the following split short exact sequence
\[0\rightarrow H^4(K_2\times K(\Z,2))\xrightarrow{h^*}H^4(Y)\rightarrow\Z\rightarrow 0.\]
In particular, the image of $h^*$ contains all torsion elements of $H^4(Y)$.
\item
Let $\zeta_2'\in H^2(K_2;\Z/2)$ be the fundamental class, and let
\[\rho_2: H^*(-;\Z)\rightarrow H^*(-;\Z/2)\]
be the $\! \! \! \pmod {2}$ reduction map. Then the class
\[h^*((\zeta_2')^2\times1-\zeta_2'\times\rho_2(\iota_2))\in H^4(Y;\Z/2)\]
has an integral lift $\nu\in H^4(Y)$, i.e., we have
\[\rho_2(a)=h^*((\zeta_2')^2\times1-\zeta_2'\times\rho_2(\iota_2)).\]
\item
\[\opn{Sq}^3_{\Z}\cdot\rho_2(\nu)=0.\]
\end{enumerate}
\end{lemma}
\begin{proof}
By definition we have a homotopy fiber sequence
\[K(\Z,4)\rightarrow Y\xrightarrow{h}K_2\times K(\Z,2).\]
We denote by $E_*^{*,*}(Y)$ the associated cohomological integral Serre spectral sequence. The statement (1) then follows immediately from $E_2^{*,*}(Y)$.

For the statement (2), notice that by the definition of $Y$ we have
\[h^*\beta^{\Z/2}((\zeta_2')^2\times1-\zeta_2'\times\rho_2(\iota_2))=h^*(2Q_2(\zeta_2')\times1-\zeta_2\times\iota_2)=0,\]
and the statement (2) follows from the long exact sequence of cohomology groups induced from the short exact sequence of coefficients gourps
\[0\rightarrow\Z\xrightarrow{\times2}\Z\rightarrow\Z/2\rightarrow0.\]

Statement (3) follows from direct computation
\begin{equation*}
\begin{split}
 &\opn{Sq}^3_{\Z}\cdot\rho_2(\nu)\\
=&\beta^{\Z/2}\cdot\opn{Sq}^2\cdot h^*((\zeta_2')^2\times1-\zeta_2'\times\rho_2(\iota_2))\\
=&\beta^{\Z/2}\cdot h^*(\zeta_2^2\times1-(\zeta_2')^2\times\rho_2(\iota_2)-\zeta_2'\times\rho_2(\iota_2)^2)\\
=&h^*(2Q_2\times\iota_2-\zeta_2\times_2^2)\\
=&h^*[(2Q_2\times1-\zeta_2\times\iota_2)(1\times\iota_2)]=0.
\end{split}
\end{equation*}
\end{proof}

\begin{proof}[Proof of Theorem \ref{thm:not(2,2,4)}]
First we recall some general facts on the Atiyah-Hirzebruch spectral sequences of both the twisted and untwisted $K$-theories. Let $(\widetilde{E}_*^{*,*},\widetilde{d}_*^{*,*})$ be the twisted Atiyah-Hirzebruch spectral sequence for a CW complex $X$ and a class $\alpha\in TH^3(X)$ of period $n$. Then $(\widetilde{E}_*^{*,*},\widetilde{d}_*^{*,*})$ is a bi-graded module over $(E_*^{*,*}, d_*^{*,*})$. More precisely, suppose $u\in \wt{E}_r^{s,t}$ and $v\in \widetilde{E}_r^{s',t'}$. Then the module structure yields  $uv\in\widetilde{E}_r^{s+s',t+t'}$ and furthermore,
\begin{equation}\label{eq:module structure1}
\widetilde{d}_r(uv)=d_r(u)v+(-1)^{s+t}u\widetilde{d}_r(v).
\end{equation}
It follows from the Bott periodicity theorem that $E_{2i}^{*,*}=E_{2i+1}^{*,*}$ for all $i>0$. The Bott periodicity theorem also indicates
\begin{equation*}
\wt{E}_2^{s,t}\cong
\begin{cases}
H^s(X)\textrm{ $t$ even},\\
0,\textrm{ $t$ odd}.
\end{cases}
\end{equation*}
In view of the above, we identify the group $\wt{E}^{s,2t}_r$ with the corresponding subquotient of $H^s(X)$.
For $u\in\wt{E}_2^{*,*}=E_3^{*,*}$, by the discussion in 2.3 of \cite{atiyah1961vector}, we have $d_3(u)=\opn{Sq}_{\Z}^3(u)$, where $\opn{Sq}^3_{\Z}$ be the $3$rd integral Steenrod square operation. On the other hand, by Proposition 2.4 of \cite{antieau2014period}, we have $\widetilde{d}_3(1)=\alpha$. Therefore, for $y\in\widetilde{E}_2^{s,t}$, we have
\begin{equation}\label{eq:module structure2}
\begin{split}
&\widetilde{d}_3(y)=\widetilde{d}_3(y\cdot 1)\\
=&\opn{Sq}_{\Z}^3(y)+(-1)^{s+t}y\widetilde{d}_3(1)\\
=&\opn{Sq}_{\Z}^3(y)+(-1)^{s+t}y\alpha.
\end{split}
\end{equation}
It is considerably harder to obtain a similar formula for $\wt{d}_5$ since it involves unstable cohomology operations. Nonetheless we may compute $\wt{d}_5$ for
\[n\times\iota_2\in H^2(K_n\times K(\Z,2)).\]
Indeed, by (\ref{eq:module structure2}) we have $\wt{d}_3(n\times\iota_2)=0$, and since $K(\Z,2)$ has trivial cohomology in odd degrees, we have $\wt{d}_3(1\times\iota_2)=0$ and $\wt{d}_5(1\times\iota_2)=0$. Now it follows from(\ref{eq:module structure1}) that we have
\begin{equation}\label{eq:module structure3}
\begin{split}
&\wt{d}_5(n\times\iota_2)\\
=&\wt{d}_5(n)1(\times\iota_2)+n\wt{d}_5(1\times\iota_2)\\
=&Q_n\times\iota_2.
\end{split}
\end{equation}
By Lemma \ref{lem:universal Y}, it suffices to verify the proposition in the case $X=Y$ and
\[\alpha=\alpha_Y:=h^*(\zeta_2\times1).\]

Let $\xi_Y:=h^*(\zeta_2'\times1)$. It follows from Theorem \ref{thm:tpip 8-cplex'} that we have $\widetilde{d}_7^{0,0}(8)=[R_2(\xi_Y)]$, where $[R_2(\xi_Y)]\in\wt{E}_7^{7,-6}$ denotes the equivalence class of $R_2(\xi_Y)$. Therefore, it suffices to show $[R_2(\xi_Y)]=0\in\wt{E}_7^{7,-6}$.

By Lemma \ref{lem:cohomology of Y}, we have a class $\nu\in H^4(Y)$ such that
\[\rho_2(\nu)=h^*((\zeta_2')^2\times1-\zeta_2'\times\rho_2(\iota_2)).\]
By (\ref{eq:Pontryagin powers}), we have
\begin{equation*}
\begin{split}
&R_2(\xi_Y)=\beta^{\Z/2}((\xi_Y)^3)\\
=&\beta^{\Z/2}[\xi_Y\rho_2(\nu)+h^*((\zeta_2')^2\times\rho_2(\iota_2))]\\
=&\alpha_Y\nu+h^*(2Q_2\times\iota_2)\\
=&\alpha_Y\nu+Q_2(\alpha_Y')2h^*(1\times\iota_2).
\end{split}
\end{equation*}
Applying (3) Lemma \ref{lem:cohomology of Y} to the above, we have
\begin{equation}\label{eq:R_2pre}
\begin{split}
&R_2(\xi_Y)=\beta^{\Z/2}((\xi_Y)^3)\\
=&[\alpha_Y\nu+\opn{Sq}_{\Z}^3(\nu)]+Q_2(\alpha_Y')2h^*(1\times\iota_2).
\end{split}
\end{equation}
Applying (\ref{eq:module structure2}) and (\ref{eq:module structure3}) to (\ref{eq:R_2pre}), we have
\begin{equation}\label{eq:R_2}
\begin{split}
&R_2(\xi_Y)=\beta^{\Z/2}((\xi_Y)^3)\\
=&[\alpha_Y\nu+\opn{Sq}_{\Z}^3(\nu)]+Q_2(\alpha_Y')2h^*(1\times\iota_2)\\
=&\wt{d}_3(\nu)+2\wt{d}_5(2h^*(1\times\iota_2)),
\end{split}
\end{equation}
which proves $R_2(\alpha_Y')=0$, and we are done.
\end{proof}

\section{Examples} \label{sec:egs}
In this section we give examples of closed, connected, orientable $8$-manifolds $M$
with Brauer classes $\alpha \in TH^3(M)$ having interesting period vectors,
thereby proving Theorem \ref{thm:egs}.

\begin{proposition} \label{prop:ceg}
There is a closed, connected orientable, non-spin$^c$ $8$-manifold $M$ with Brauer class $\alpha \in TH^3(M)$ such that
$\perv(\alpha) = (2, 4, 2)$.  
It follows that $\ind(\alpha) = \per(\alpha)^4$ and so $\alpha$ violates
the TPIC.
\end{proposition}

\begin{remark}
We currently do not know of any spin$^c$ counter examples to the TPIC.
\end{remark}

\begin{proof}
We start with $N$ an orientable but non-spin$^c$ $6$-manifold and $\beta \in TH^3(N)$ with
$\perv(\beta) = (2, 4)$.
Such pairs $(N, \beta)$ exist by
By \cite[Theorem 1.4(2)]{crowley2018topological}.
Set
$$ M_0 := S^2 \times N,$$
let $\pi_N \colon M_0 \to N$ and $\pi_{S^2} \colon S^2 \times N \to S^2$
be the projections and $\alpha_0 := \pi_N^*(\beta)$.
We have $\pi_2(M_0) \cong \pi_2(S^2) \times \pi_2(N) \cong \Z \oplus \pi_2(N)$.  By general position,
the homotopy class $(2, 0)$ is represented by an embedding $\bar \phi \colon S^2 \to M_0$, which
is unique up to isotopy.   Since the stable tangent bundle
of $M_0$ restricted to $S^2 \times \ast$ is trivial, it follows that the normal bundle of $\bar \phi$ is
trivial.  Hence there is an embedding
$$ \phi \colon D^6 \times S^2 \to M_0,$$
which represents $(2, 0) \in \pi_2(M_0)$ and which is unique up to isotopy since $\pi_2(SO(6)) = 0$.
We define $M$ to be the outcome of surgery on $\phi$.  That is,
the compact manifold
$$ W_\phi := (M_0 \times I) \cup_\phi (D^6 \times D^3)$$
has boundary $M_0 \sqcup M$.
Below we shall see that there is a canonical isomorphism
$$ TH^3(M) \cong \Z/2 \oplus TH^3(M_0).$$
We set $\alpha := ([1], \alpha_0) \in TH^3(M)$ and we will show that
$\perv(\alpha) = (2, 4, 2)$ and that $M$ is not spin$^c$.

The bordism $W$ has homotopy type given by
$$ W \simeq M_0 \cup_{\bar \phi} D^3 \simeq M \cup D^6.$$
For $u \in H^2(M_0)$ the pull back of a generator of $H^2(S^2)$ along $\pi_{S^2}$,
the Kunneth Theorem gives
$$H^i(M_0) = H^i(N) \oplus  u H^{i-2}(N),$$
where we have dropped $\pi_N^*$ from the notation, identifying $H^*(N)$
with its image under $\pi_N^*$.
If $i_0 \colon M_0 \to W$ is the inclusion then an elementary calculation using the
homotopy equivalence $W \simeq M_0 \cup_{\bar \phi} D^3$ gives
\begin{equation} \label{eq:WZ}
i_0^* \colon H^*(W) \to H^*(M_0)
\begin{cases}
\text{is a 
split injection onto $H^2(N) \subset H^2(M_0)$} & \ast = 2,\\
\text{is a 
split surjection with kernel $\Z/2$} & \ast = 3, \\
\text{is an isomorphism} & \ast \neq 2, 3 \\
\end{cases}
\end{equation}
and
\begin{equation} \label{eq:WZ2}
i_0^* \colon H^*(W; \Z/2) \to H^*(M_0; \Z/2)
\begin{cases}
\text{is a split surjection with kernel $\Z/2$} & \ast = 3, \\
\text{is an isomorphism} & \ast \neq 3. \\
\end{cases}
\end{equation}
\noindent
On the other hand, we have the homotopy equivalence $W \simeq M \cup D^6$, 
where rationally $H^6(W; \Q) \cong H^6(M; \Q) \oplus \Q$.
It follows that for the inclusion
$i \colon M \to W$ we have
\begin{equation} \label{eq:i}
i^* \colon H^*(W) \to H^*(M)
\begin{cases}
\text{is a split surjection with kernel $\Z$} & \ast = 6, \\
\text{is an isomorphism} & \ast \neq 6. \\
\end{cases}
\end{equation}
%
\noindent
In particular, for $j \neq 2, 3$, we have the zig-zag of homomorphisms
$$ H^j(M_0) \xleftarrow{~i_0^*~} H^j(W) \xrightarrow{~i^*~} H^j(M)$$
and for $j \neq 2, 3$ we define
$$ \psi^j := i^* \circ (i_0^*)^{-1} \colon H^j(M_0) \to H^j(M),$$
which is an isomorphism if in addition $j \neq 6$.

Now $\alpha_0 = \beta^{\Z/2}(\xi'_0)$ for some $\xi'_0 \in H^2(N; \Z/2)$ and set
$$\xi_0 := (\rho_2(u), \xi'_0) \in H^2(M_0; \Z/2) = \Z/2(\rho_2(u)) \oplus H^2(N; \Z/2),$$
where $\rho_2$ denotes reduction mod~$2$.
By \eqref{eq:WZ2}, there is a unique class $\xi_W \in H^2(W;\Z/2)$ such that $i_0^*(\xi_W) = \xi_0$
and we set $\alpha_W := \beta^{\Z/2}(\xi_W)$ and $\xi := i^*(\xi_W)$.  Then it is easy 
check that under the isomorphisms
$$  TH^3(M) \cong TH^3(W) \cong \Z/2 \oplus TH^3(M_0),$$
the element $\alpha = ([1], \alpha_0)$ satisfies $\alpha = \beta^{\Z/2}(\xi) \in TH^3(M)$.
Of course, $\alpha_0 = i_0^*(\alpha_W)$ and $\alpha=i^*(\alpha_W)$.

It remains to compute $\ord(G_1(\alpha))$ and $\ord(G_2(\alpha))$.
For $G_1(\alpha)$ we use the formula \eqref{eq:G_1}, which gives
$G_1(\alpha) = \beta^{\Z/2}(P_2(\xi^2))$.
Since $i^* \colon H^j(W) \to H^j(M)$ is an isomorphism for $j = 2, 5$, it follows that
$$ \ord(G_1(\alpha)) = \ord(G_1(\alpha_W)).$$
Now can compute $G_1(\alpha_W)$ by restricting to $M_0$.  Specifically,
$$i_0^*(G_1(\alpha_W)) = G_1(i_0^*(\alpha_W)) = 
[\beta^{\Z/2}((\rho_2(u) + \xi_0)^2)]
= [\beta^{\Z/2}(\xi_0^2)].$$
Since $H^2(M_0) \alpha \cap H^5(N) = H^2(N) \beta$, it follows that
$$ \ord(G_1(\alpha_W)) = \ord(G_1(\alpha_0)) = \ord(G_1(\beta)) = 4$$
and so $\ord(G_1(\alpha)) = 4$.

To compute $G_2(\alpha)$ we need to compute $R_2(\xi)$.
Now $R_2(\xi) = \psi^7(R_2(\xi_0))$ and
%
\begin{multline} \label{eq:R_2xi}
 R_2(\xi_0) =  \beta^{\Z/2}((\rho_2(u) + \xi'_0)^3) = \\
 \beta^{\Z/2}(\rho_2(u)^3 + 3\rho_2(u)^2\xi'_0 + 3\rho_2(u)(\xi'_0)^2 + (\xi'_0)^3) = 
 \beta^{\Z/2}(\rho_2(u) (\xi'_0)^2) = u\beta^{\Z/2}((\xi'_0)^2)\neq 0.
\end{multline}
%
Since $\psi^7$ is an isomorphism, $R_2(\xi) \neq 0$ and it remains to show that $R_2(\xi)$ does not lie in the
indeterminacy of $G_2$, which for a space $X$ is the group
%
\begin{equation}\label{eq:indet}
I_2(X) = (Sq^3_{\Z} + \cup \alpha)H^4(X) + \opn{Im}\wt{d}_5^{2,-2}(H^2(X)) \subset H^7(X).
\end{equation}%
We note that 
$I_2(M) \subseteq \psi(I_2(M_0))$ and we can use naturality to compute $I_2(M)$ using $I_2(M_0)$.
We have that
$$ H^7(M_0) = H^7(N) \oplus uH^5(N) = uH^5(N)$$
and $H^4(M_0) = H^4(N) \oplus u H^2(N)$. Now $Sq^3_{\Z}(H^4(N)) = 0$
and $\opn{Sq}^3_{\Z}(u H^2(N)) = 0$, since $\opn{Sq}^3_{\Z}$ vanishes on
any product of two degree two classes.  
%
%
Therefore we conclude
\begin{equation}\label{eq:Sq^3_Z(H^4)}
\opn{Sq}^3_{\Z}(H^4(M_0))=0.
\end{equation}
Now $[\beta^{\Z/2}((\xi_0')^2)] \neq 0 \in H^5(N)/\beta H^2(N)$ 
and by (\ref{eq:R_2xi}) 
\[R_2(\xi_0)=u\beta^{\Z}((\xi_0')^2).\]
However, since $\alpha_0 = \pi_N^*(\beta) \in H^3(M_0)$,
\[ \alpha_0H^4(M) = \alpha_0 \bigl( H^4(N) \oplus y H^2(N) \bigr) = u \beta H^2(N)\]
and so
\[R_2(\xi_0)\notin \alpha_0 H^4(M_0).\]
This and (\ref{eq:Sq^3_Z(H^4)}) show that $R_2(\xi_0)$ does not lie in 
$(\opn{Sq}^3_{\Z} + \cup \alpha)(H^4(M_0))$ 
and hence $R_2(\xi_W)$ does not lie in $(\opn{Sq}^3_{\Z} + \cup \alpha_W)(H^4(W))$.

Now $i_0^* \colon H^2(W) \to H^2(M_0)$ is an isomorphism with
$i_0^*(H^2(W)) = H^2(N)$ and $G_2(H^2(N)) = 0$ since $H^7(N) = 0$.
By naturality, it follows that $\wt{d}_5^{2,-2}(W) = 0$.
Comparing the above to (\ref{eq:indet}), it follows that $R_2(\xi_W)$ does not lie in $I_2(W)$
and so $R_2(\xi)$ does not lie in $I_2(M)$.

We must finally show that $M$ is not spin$^c$.  Recall that a manifold $X$ is not spin$^c$ if 
and only if $W_3(X) \neq 0 \in H^3(X)$.
By construction $M_0$ is not spin$^c$ and so $W_3(M_0) \neq 0$.
Since $W_3$ is a stable characteristic class, $W_3(M_0) = i_0^*(W_3(W))$
and so $W_3(W) \neq 0$.
Now $W_3(M) = i^*(W_3(W))$ and by \eqref{eq:i}, $i^* \colon H^3(W) \to H^3(M)$
is an isomorphism.  Hence $W_3(M) \neq 0$ and $M$ is not spin$^c$.
\end{proof}

\begin{proposition} \label{prop:2-4-1}
There is a closed, connected parallelisable (hence almost complex)
$8$-manifold $M$ with Brauer class $\alpha \in TH^3(M)$ such that
$\perv(\alpha) = (2, 4, 1)$.
\end{proposition}

\begin{proof}
Choose a finite $5$-skeleton $K \subset K_2$ for $K_2$.
By a theorem of Wall \cite[Embedding Theorem]{Wall1966}, there is a {\em thickening} of $K$
(i.e.\ a homotopy equivalence $\phi \colon K \to W$, where $W$ is compact
smooth manifold with simply-connected boundary) such that $W \subset S^9$
is a co-dimension zero submanifold of the $9$-sphere.
We let $M_0 := \del W$ be the boundary of $W$ and $C \subset S^9$
be the complement of the interior of $W$, so that $S^9 = W \cup_{M_0} C$.

By Alexander duality, $C$ is $2$-connected.
The Mayer-Vietoris sequence for the decomposition $S^9 = W \cup_M C$ gives that
the sum of the maps induced by the inclusions $i_W \colon M \to W$ and $i_C \colon M_0 \to C$
$$ i_{W*} \oplus i_{C*} \colon H_i(M_0) \to H_i(W) \oplus H_i(C)$$
is an isomorphism for $i = 1, \dots, 7$.  Applying the Universal Coefficient Theorem
we deduce that we have isomorphisms
$$ TH^3(M_0) \xra{i_W^*} TH^3(W) \cong TH^3(K) = TH^3(K_2) = \Z/2$$
and
$$H^2(M_0) \cong H^2(W) \cong H^2(K) = 0.$$
We also deduce that $i_W^* \colon TH^5(M_0) \to TH^5(W) \cong TH_4(K) = TH_4(K_2) \cong \Z/4$
is a split surjection.

Let $\alpha_0 \in TH^3(M_0)$ be the generator.  The above shows that $G_1(\alpha_0)$
generates a $\Z/4$-summand of $TH^5(M_0)/\alpha H^2(M_0) = TH^5(M_0)$.
Since $\alpha_0$ pulls back from $W \simeq K$ and $K$ is $5$-dimensional,
the higher periods of $\alpha_0$ are $1$ by naturality and so $\perv(\alpha_0) = (2, 4, 1)$.

Since $M_0$ is a hyper-surface in $S^9$, $M_0$ is stably paralellisable.  It follows that
the tangent bundle of $M_0$ pulls back along the degree one collapse map $c \colon M_0 \to S^8$
and that there is a vector bundle isomorphism
$$ TM_0 \cong c^*(k\tau_{S^8}),$$
where we regard $\tau_{S^8}$ as an element of $\pi_7(SO(8))$ and $k$ is some integer 
determined by the equation
$$ \chi(M_0) = 2k.$$
It follows that we may take the connected sum of $M_0$ and copies of either $S^4 \times S^4$
or $S^3 \times S^5$, to achieve a parallelisable manifold $M$ with $TH^3(M_0)$ identified with $TH^3(M)$.  Moreover, if we take $\alpha \in TH^3(M)$ corresponding to $\alpha_0 \in TH^3(M_0)$ 
under this identification, then it is clear that $\perv(\alpha) = \perv(\alpha_0)$.
This finishes the proof.
\end{proof}

\bibliographystyle{abbrv}
\bibliography{cghref}
\end{document}